\documentclass{amsart}

\usepackage{enumerate,mathtools,mathptmx,xcolor,bbold,subcaption,amssymb,amsmath}
\usepackage[colorlinks=true,linkcolor=red!70!black,citecolor=green!70!black,urlcolor=magenta!70!black]{hyperref}
\usepackage{tikz}
\usepackage{tikz-cd}

\newtheorem{mainthm}{Theorem}
\newtheorem{thm}{Theorem}[subsection]
\newtheorem{lem}[thm]{Lemma}

\newtheorem{cor}[thm]{Corollary}

\newtheorem{rmk}[thm]{Remark}

\newcommand{\defeq}{\stackrel{\mathrm{def}}{=}}
\newcommand{\NN}{\mathbb{N}}

\newcommand{\ZZ}{\mathbb{Z}}

\newcommand{\CC}{\mathbb{C}}
\newcommand{\p}{\mathsf{P}}
\newcommand{\NP}{\mathsf{NP}}
\newcommand{\shP}{\#\mathsf{P}}
\newcommand{\fP}{\mathsf{fP}}

\newcommand{\calR}{\mathcal{R}}

\DeclareMathOperator{\CSP}{CSP}

\DeclareMathOperator{\Hom}{Hom}

\begin{document}

%%%%%%%%%%%%%%%%%%%%%%%%%%%%%%%%%%%%%%%%%%%%%%%%%%%%%%%%%%%%%%%%%%%%%%%%%%%%%%%%
% Title, metadata, etc...

\title[Obstruction theory and counting complexiy]{Obstruction theory and the complexity of counting group homomorphisms} 

\author{Eric Samperton}
\address{
	Departments of Mathematics and Computer Science \\
	Purdue Quantum Science \& Engineering Institute \\
	150 North University Street \\
	Purdue University \\
	West Lafayette, IN \\
	47907
}
\email{eric@purdue.edu}

\author{Armin Wei{\ss} }
\address{FMI, University of Stuttgart \\
	 Universitätsstraße 38\\
	  70569 Stuttgart, Germany }
\email{armin.weiss@fmi.uni-stuttgart.de}

\date{\today}

\begin{abstract}
	Fix a finite group $G$.
	We study the computational complexity of counting problems of the following flavor: given a group $\Gamma$, count the number of homomorphisms $\Gamma \to G$.
	Our first result establishes that this problem is $\#\mathsf{P}$-hard whenever $G$ is a non-abelian group and $\Gamma$ is provided via a finite presentation.
	We give several improvements showing that this hardness conclusion continues to hold for restricted $\Gamma$ satisfying various promises.
	Our second result shows that if $G$ is class 2 nilpotent and $\Gamma = \pi_1(M^3)$ for some input 3-manifold triangulation $M^3$ with $|H^2(M,Z(G)|$ bounded above, then there is a polynomial time algorithm to compute the number of homomorphisms from $\Gamma$ to $G$.		
	This algorithm is explained in part by the fact that 3-manifolds are close enough to being Eilenberg-MacLane spaces for us to solve the necessary group cohomological obstruction problems efficiently using the given triangulation.
	A similar polynomial time algorithm for counting maps to finite, class 2 nilpotent $G$ exists when $\Gamma$ is itself a finite group encoded via a multiplication table, provided that $|H^2(\Gamma,Z(G))|$ is similarly bounded from above.
\end{abstract}

\thanks{
ES supported in part by NSF CCF 2330130.
AW supported by Deutsche Forschungsgemeinschaft (DFG, German Research Foundation) grant WE~6835/1-2.
ES thanks Greg Kuperberg, Nathan Dunfield, Colleen Delaney, Ben McReynolds, Eric Rowell, and Susan Hermiller for various helpful conversations.
Both authors thank Murray Elder for helpful conversations, as well as for introducing the authors to each other.
}

\maketitle

% Main Body
%%%%%%%%%%%%%%%%%%%%%%%%%%%%%%%%%%%%%%%%%%%%%%%%%%%%%%%%%%%%%%%%%%%%%%%%%%%%%%%%
\section{Introduction}
\label{sec:intro}
\subsection{Main results}
\label{ss:results}
The goal of this paper is to better understand the computational problem where we count the homomorphisms from an input group $\Gamma$ to a fixed finite target group $G$:
\[ \#\Hom(-,G) :\; \Gamma \mapsto \#\Hom(\Gamma, G) \defeq \#\{F: \Gamma \to G \mid F \text{ a group homomorphism}\} \in \NN.\]
We are especially interested in how the encoding of $\Gamma$ affects the complexity of this problem.
We have two main theorems.

Our first result shows that when $\Gamma$ is provided as a finitely presented group, the homomorphism counting problem for fixed target groups $G$ satisfies a hardness dichotomy according to whether or not $G$ is abelian.
\begin{mainthm}
	\label{thm:1}
	Fix a finite group $G$.
	The counting problem that takes a finitely presented group $\Gamma$ to the number of homomorphisms $\#\Hom(\Gamma,G)$ is $\shP$-complete via Turing reduction if and only if $G$ is nonabelian.
	If $G$ is abelian, then the problem is in $\fP$ (\emph{i.e.}~ it is solvable in polynomial time).
\end{mainthm}
We note that we take the usual notion of size for a finitely presented group: the sum of its number of generators and the lengths of all of its relations.

As stated, we suspect Theorem \ref{thm:1} might be well known to experts, cf.\ \cite[Lem.~3.8]{larose} or \cite[Thm.~8]{count}.
In particular, similar to what we do in our proof, Nordh and Jonsson use results of Bulatov and Dalmau \cite{BulatovDalmau:dichotomy} to show that a more general problem---namely, counting solutions to systems of equations over $G$---satisfies the same dichotomy \cite{count}; see the end of Subsection \ref{ss:related} for further discussion.

In the details, our proof is a little bit different from Nordh and Jonsson: compare our Lemma \ref{lem:maltsev} with \cite[Thm.~8]{count}.
Moreover, while on one hand our Lemma \ref{lem:maltsevCommutator} and its proof are both essentially paraphrasings of \cite[Thm.~8]{count}, on the other hand,
we use the lemma to arrive at a stronger conclusion: homomorphisms (not just solutions to systems of ``commutation equations" with constants) from right-angled Artin groups to non-abelian $G$ are $\shP$-hard to count.
In Subsection \ref{ss:restrictedInput}, we give two more variations/modifications of our proof of Theorem \ref{thm:1} that lead to other classes of groups $\Gamma$ where counting maps to non-abelian $G$ is hard.
We summarize these three results in the next corollary.

\begin{cor}\label{cor:restrictedInput}
	For a non-abelian group $G$, the hardness conclusion of Theorem~\ref{thm:1} still holds if we promise the input finitely-presented group $\Gamma$ is taken from any one of the following restricted classes:
	\begin{itemize}
		\item right-angled Artin groups (RAAGs),
		\item torsion-free groups satisfying the small-cancellation condition $C'(1/6)$ (in particular, hyperbolic groups),
		\item finite nilpotent groups encoded either via finite presentations or as finite permutation groups, if we furthermore assume that $G$ itself is non-abelian nilpotent.
	\end{itemize}
\end{cor}

For the definitions of right-angled Artin groups and the $C'(1/6)$ condition, see Subsection \ref{ss:restrictedInput}.
Note that groups satisfying the $C'(1/6)$ condition have a word problem that is algorithmically solvable in linear time (via Dehn's algorithm).
RAAGs also have an efficiently solvable word problem \cite{charney,Susan}.
Thus, Corollary \ref{cor:restrictedInput} clarifies Theorem \ref{thm:1} by showing that the hardness of counting maps to a finite group $G$ is not secretly a repackaging of the algorithmic insolubility of the word problem for finitely presented groups.

The dichotomy formulated in Theorem \ref{thm:1} is a warm-up to the kind of dichotomy result we hope will eventually be proved for invariants of 3-dimensional manifolds determined by topological quantum field theories \cite{me:dichotomy}.
We briefly review these motivations in Section \ref{ss:motivations}, but for now, suffice it to say that we would like to know to what extent a result like Theorem \ref{thm:1} might continue to hold if we constrained the input finitely presented group to be of the form $\Gamma = \pi_1(M)$ where $M$ is an orientable, closed 3-manifold triangulation.
Our second main result shows that any similar such dichotomy for such restricted $\Gamma$ will necessarily require that we allow 3-manifolds with ``large'' second cohomology groups.

\begin{mainthm}
	\label{thm:2}
	Fix a finite nilpotent group $G$ of class 2 and some positive constant $C$.
	Let $Z(G)$ be the center of $G$.
	If $M$ is a triangulated, orientable, closed 3-manifold with $|H^2(M,Z(G))|<C$, then the counting problem $M \mapsto \#\Hom(\pi_1(M),G)$ is in $\fP$ (\emph{i.e.}~ it is solvable in polynomial time).
\end{mainthm}

\begin{rmk}
   We are aware that, while the previous version of this theorem
contained a mistake, the current version can be proven in an easier
way. We will address this in a future version of this paper.
\end{rmk}

Here the size of a triangulated 3-manifold $M$ is the total number of all of its simplices.
It is well-known that given $M$ one can prepare a presentation of $\pi_1(M)$ in polynomial time.
We contrast Theorem \ref{thm:2} with a prior result proved by the first author and Greg Kuperberg: counting homomorphisms from triangulated 3-manifolds' fundamental groups to a finite non-abelian simple group is $\shP$-hard under almost parsimonious reduction, even if we restrict to 3-manifolds with trivial second and first (co)homology (that is, even if we restrict to \emph{integer homology 3-spheres}) \cite{me:zombies}.

In Subsection \ref{ss:table}, we discuss how a version of  Theorem \ref{thm:2} also holds for finite groups $\Gamma$ encoded via multiplication tables.

Given Theorems \ref{thm:1} and \ref{thm:2}, it is natural to wonder if, even for class 2 nilpotent groups $G$, the counting problem
\[ M^3 \mapsto \#\Hom(\pi_1(M), G)\]
is $\shP$-hard.
We hope to address this question in future work.
In \S\ref{ss:dihedral}, we report briefly on some related work in progress.

\subsection{Proof sketch for Theorem 1}
\label{ss:sketch}
The proof of Theorem \ref{thm:1} is a straightforward application of well-established dichotomy techniques for generalized counting constraint satisfaction problems ($\#$CSPs) over not-necessarily-binary alphabets, due to Bulatov and Dalmau \cite{BulatovDalmau:dichotomy,Bulatov:dichotomy} (compare with \cite{CreignouHermann}, which gives a dichotomy over binary alphabets).
In brief, for a fixed finite target group $G$, we define a specific ternary relation $T_G$ over the alphabet $G$, which is essentially the multiplication table of $G$:
\[ T_G \defeq \{(g,h,gh) \mid g,h\in G\} \subseteq G\times G \times G.\] 
By padding with additional generators and using substitution, a finite presentation $\Gamma = \langle x_1,\dots,x_n \mid r_1,\dots,r_m\rangle$ can be assumed to have all relations of the form $x_{i_1}x_{i_2}=x_{i_3}$.
Then a homomorphism $\Gamma \to G$ is precisely the same thing as an assignment of generators $x_i$ to elements $g_i$ of $G$ so that $(g_{i_1},g_{i_2},g_{i_3}) \in T_G$ for each relation $r_i$---in other words, a solution to an instance of CSP$(\{T_G\})$.
By \cite[Thm.~5]{BulatovDalmau:dichotomy}, the counting problem $\#$CSP$(\{T_G\})$ is $\shP$-hard (via Turing reduction) if $T_G$ does not support a Mal'tsev polymorphism.
One concludes with an elementary argument that $T_G$ has a Mal'tsev polymorphism if and only if $G$ is abelian.
We provide the full details in Section \ref{sec:thm1}, along with variations on this proof that establish Corollary \ref{cor:restrictedInput}.
We note that the proof for RAAGs we provide is arguably even simpler than that for general finitely presented groups (notwithstanding that the RAAG case implies the general case); however, as our two proofs require applying the results of Bulatov and Dalmau in slightly different ways, we find it instructive to include both.

\subsection{The algorithmic strategy for Theorem 2}
\label{ss:strategy}
Two general ideas underpin the algorithm necessary to establish Theorem \ref{thm:2}.
First, we can always reduce the problem of counting maps with target a class 2 nilpotent group $G$ to obstruction-theoretic problems in group cohomology.\footnote{In fact, this kind of idea was explored for the more general problem of counting maps to finite \emph{solvable} $G$ in \cite{MateiSuciu}, but no complexity bounds were provided.  In a similar vein, \cite{Chen} asserts that for central extensions built from finite abelian groups $A$ and $Q$ there is an ``efficient" criterion for testing if an endomorphism of $Q$ lifts to a homomorphism of central extensions.
}
Second, if $\Gamma$ is encoded in a manner that allows us efficient access to an explicit dense model for its Eilenberg-MacLane classifying space $BG = K(G,1)$ (or at least its 3-skeleton) and $H^2(\Gamma,Z(G))$ is not too large, then we can solve these obstruction problems efficiently using ``Fourier analysis" over finite abelian groups; see Theorem \ref{thm:skeleton} for a precise statement.

Let us expand on these points now.

If $G$ is any (finite) nilpotent group of class 2 with center $A\defeq Z(G)$, then $G$ is determined (up to isomorphism) by the abelian groups $A$ and $B\defeq G/A$, and a normalized 2-cocycle $\eta \in Z^2(B, A)$.
If we write the group operations of $A$ and $B$ additively, this means $\eta: B \times B \to A$ satisfies the cocycle equation
\[\eta(b_1,b_2)-\eta(b_1,b_2+b_3)+\eta(b_1+b_2,b_3)-\eta(b_2,b_3) = 0 \qquad (\forall b_1,b_2,b_3 \in B) \]
and is normalized in the following sense:
\[ \eta(b,0)=\eta(0,b) = 0 \qquad (\forall b \in B) .\]
In terms of this data, we can identify $G$ with the \emph{set} $A\times B$ equipped with the ``$\eta$-deformed" binary operation
\[ (a_1,b_1)\cdot (a_2,b_2) = (a_1 + a_2 + \eta(b_1,b_2), b_1+b_2).\]

Any homomorphism $\phi: \Gamma \to G$ induces a homomorphism $\overline{\phi} \defeq \pi_A \circ \phi: \Gamma \to B$, where $\pi_A: G \to B$ projects out $A$.  Thus,
\[ \#\Hom(\Gamma, G) = \sum_{\psi \in \Hom(\Gamma, B)} \#\{\phi \in \Hom(\Gamma, G) \mid \overline{\phi} = \psi\}\]
where $\Hom(\Gamma, B)$ is a finite abelian group because $B$ itself is.
By well-known properties of group cohomology, a given homomorphism $\psi: \Gamma \to B$ admits a lift to a homomorphism $\phi: \Gamma \to G$ with $\overline{\phi}=\psi$ if and only if the pullback cohomology class $[\psi^*\eta]=0 \in H^2(\Gamma,A)$ vanishes in the cohomology of $\Gamma$ with coefficients in $A$.
Critically, because $G$ is a \emph{central} extension of $B$ by $A$, anytime $A$ appears as the coefficients in a cohomology group, it is with a \emph{trivial} module structure.
This implies that we have a well-defined set-theoretic map
\[
	\begin{aligned}
		(-)^*: \Hom(\Gamma,B) &\to H^2(\Gamma,A) \\
		\psi &\mapsto \psi^*[\eta] = [\psi^*\eta],
	\end{aligned}
\]
which is \emph{not} a homomorphism typically, but rather, a \emph{quadratic} map; see Lemma \ref{lem:quad}.

Moreover, if any lift of $\psi$ exists, then (again because $G$ is a central extension) the set of all such is a torsor for $H^1(\Gamma,A) = \Hom(\Gamma, A)$, which is itself an abelian group.\footnote{
Recall that a \emph{torsor} (or \emph{affine space)} for a group $H$ is a set $X$ together with an action of $H$ on $X$ that is both free and transitive.  
In particular, if $X$ is an $H$-torsor then $|H|=|X|$.
}
We see then that for each liftable map $\psi: \Gamma \to B$, we get $\#\Hom(\Gamma, A)$ many lifts.
In other words,
\[ \#\Hom(\Gamma, G) = \#\Hom(\Gamma,A) \times \#\{ \psi \in \Hom(\Gamma, B) \mid [\psi^*\eta]=0 \in H^2(\Gamma, A)\}.\]
The factor $\#\Hom(\Gamma,A)$ can be computed in polynomial time for any input $\Gamma$, and so we are left with the problem of counting the number of solutions to the ``obstruction problem" $[\psi^* \eta]=0$.
This is a quadratic condition on $\psi$, but, \emph{a priori}, we do not even have a concrete understanding of $H^2(\Gamma,A)$ from the presentation $\Gamma$.
Indeed, unless we know something more about $\Gamma$, the only way to get our hands on $H^2(\Gamma,A)$ requires working with the typically infinitely-generated free $A$-module $Z^2(\Gamma,A)$ \cite[\S7.6.3]{Holt:handbook}.
Even worse, Theorem \ref{thm:1} implies this problem cannot be surmounted effectively in general.

\begin{cor}
	\footnote{We note that our methods do not establish $\NP$-hardness of the \emph{existence} problem $\#\{ \psi \mid [\psi^*\eta]=0 \in H^2(\Gamma,A)\} \stackrel{?}{>} 1$.
In fact, the dichotomy for the existence problem cannot possibly be the same as that for the counting problem.
Note that the existence problem can be rephrased to the following equivalent problem: for a fixed class 2 nilpotent group $G$ and input group $\Gamma$, decide if $\Gamma$ admits a homomorphism to $G$ whose image is not contained in the center of $G$.
When $G=D_8$ is the dihedral group of order 8, for example, a short argument shows that $\Gamma$ admits such a homomorphism to $G$ if and only if $\Gamma$ admits a  surjection to $\ZZ/2\ZZ$ (the only simple factor involved in $D_8/Z(D_8)$).
This latter question can be resolved in polynomial time (see e.g.~\cite{ElderShenWeiss}).
We leave further exploration of these matters as a possible direction for future work.}
	\label{cor:obstruction}
Fix finite abelian groups $A$ and $B$, along with a cocycle $\eta: B \times B \to A$ such that $A \times_\eta B$ is not abelian.
If $\Gamma$ is a finitely presented group, then it is $\shP$-complete to count
\[  \#\{ \psi \in \Hom(\Gamma,B) \mid [\psi^* \eta] = 0 \in H^2(\Gamma,A). \}\]
\qed
\end{cor}

This is where the second key idea enters: ``most" 3-manifolds $M$ are Eilenberg-MacLane spaces $K(\pi_1(M),1)$ for their fundamental groups $\pi_1(M)$.
In particular, if $M$ is a \emph{triangulated} 3-manifold that is Eilenberg-MacLane, then we can use the \emph{simplicial cohomology} $H^2(M,A)$ to compute the \emph{group cohomology} $H^2(\pi_1(M),A)$ in polynomial-time.  Assuming we also know $|H^2(\pi_1(M),A)| < C$, then we can furthermore count the solutions to $[\psi^*\eta]=0$ efficiently; see Theorem \ref{thm:skeleton}, which we expect will be of independent interest.\footnote{
	Example 1 in Section VIII.9 of \cite{Brown:book} gives a simpler application of such observations.
}
The only speed bump with this second idea is that not all 3-manifolds are actually Eilenberg-MacLane.
Fortunately, using techniques first introduced in \cite{me:HQFT} and recently expanded by Ennes and Maria \cite{Clem:HQFT}, we may efficiently ``pre-process" $M$ to replace it with an Eilenberg-MacLane $3$-manifold $N$ such that both $\#\Hom(\pi_1(M),G) = \#\Hom(\pi_1(N),G)$ and $\#H^2(M,A)=\#H^2(N,A)$; see Theorem \ref{thm:clem}.

\subsection{Remarks on finite domain groups encoded via multiplication tables}
\label{ss:table}
While we will refrain from stating and proving another theorem, we note that a version of Theorem \ref{thm:2} also holds for finite groups $\Gamma$ that are encoded via multiplication tables.
This is because from such encodings we may efficiently build the first three terms of the bar resolution of the trivial $\Gamma$ module $\mathbb{Z}$ efficiently.
In almost equivalent language: given that we have access to the entire table for $\Gamma$, we may build the three skeleton of the Eilenberg-MacLane space $K(\Gamma,1)$ in polynomial time.
Once one has this, Theorem \ref{thm:skeleton} can be used.

\subsection{Complexity curiosities from the dihedral group of the square}
\label{ss:dihedral}
In a follow-up collaboration with Chlopecki \cite{me:dihedral}, the first author digs further into the tension between Theorems \ref{thm:1} and \ref{thm:2} by considering the special case $G=D_8$, the dihedral group of order $8=2^3$.
Since $D_8$ is (one of) the smallest non-abelian nilpotent groups, it is natural to wonder if the difficulty of counting maps $\Gamma \to D_8$ for finitely presented $\Gamma$ is due to the difficulty of computing the universal class 2 $2$-group quotient of $\Gamma$ \cite{Holt:handbook,Nickel}, or if, instead, the difficulty is something ``more intrinsic" to class 2 nilpotent groups.\footnote{
We note that \cite{Holt:handbook,Nickel} discuss general algorithms for computing universal $p$-group quotients, but they do not carry out any complexity analysis.  Our argument in \S\ref{ss:restrictedInput} suggests it should be polynomial time. }

To this end, \cite{me:dihedral} considers the case that $\Gamma$ itself is provided in such a manner that it is obviously a \emph{finite} 2-group of class 2.
The main result is that for a certain class of such $\Gamma$ encoded by a 3-tensor over the finite field $\mathbb{F}_2$, counting maps to $D_8$ is Turing equivalent to counting solutions to \emph{systems of non-homogeneous symmetric bilinear equations over $\mathbb{F}_2$}.
Intriguingly, as far as we are aware, the complexity of this latter problem has not been studied before (either in existence or counting forms); in \cite{me:dihedral}, we prove it is hard.

We note that the closely related \emph{isomorphism} decision problem for finite class 2 $p$-groups is a good candidate for having intermediate complexity (modulo collapse), as it is equivalent in power to $\mathsf{TensorIsomorphism}$ over $\mathbb{F}_2$ \cite{GrochowQiao}.
We leave further discussion, including the connection to MinRank of 3-tensors \cite{BlaserETAL}, to \cite{me:dihedral}.

\subsection{Related problems and results}
\label{ss:related}
For a fixed, finite target group $G$, there are four interesting flavors of homomorphism problems we can consider, each of which is only turned into a precise problem after we specify the encoding of $\Gamma$ (e.g., finite presentation, $\pi_1(M^3)$, finite via Cayley table, (finite) matrix group, permutation group, \emph{etc.}):\footnote{
	One might also ask about finding/counting \emph{isomorphisms} with $G$, but since then we would need $\Gamma$ itself to be finite to have much hope for interesting complexity theory (recall ``finite" is an undecidable property of finitely presented groups), we then find ourselves in the case of what one might call the Group Recognition problem for finite groups.
This is one slice of the Group Isomorphism problem, which is a very well-studied problem with a huge literature that we shall avoid attempting to review here.}

\begin{enumerate}
	\item { Does there exist a nontrivial homomorphism to $G$?}
	\item { Does there exist a surjection to $G$?}
	\item { How many homomorphisms are there to $G$?}
	\item { How many surjections are there to $G$?}
\end{enumerate}

Whenever $G$ is abelian, all four flavors have efficient algorithmic solutions for general finite presentations $\Gamma$.

As far as we are aware, the first explicit hardness results for all 4 flavors were established by Kuperberg and the first author in \cite{me:zombies}.
That work establishes for any non-abelian simple $G$ the problems (1) and (2) are $\NP$-hard and the problems (3) and (4) are $\shP$-hard for $\Gamma = \pi_1(M^3)$ (hence, for general finitely presented $\Gamma$ as well).

Work of Elder, Shen, and the second author \cite{ElderShenWeiss} shows that problem (2) is $\NP$-hard when $\Gamma$ is finitely presented and $G=D_{2n}$ is any non-nilpotent dihedral group.
While this reduction does not seem to exhibit enough parsimony to deduce $\shP$-hardness of the analogous counting problem (4), we believe that using the same approach to devise a direct reduction from \#3-SAT to (4) can lead to an ``almost" parsimonious reduction (in the same sense of \cite{me:zombies}).

Of course, Theorem \ref{thm:1} completely solves (3) for finitely presented $\Gamma$.
However, as far as we know, similar dichotomy results for the other three flavors of problems are wide open (for all the classes of $\Gamma$ we have considered so far).
Two of the most interesting open cases are when $G=D_6=S_3$ (smallest non-nilpotent group) or $G=D_{16}$ (one of the smallest nilpotent groups of class 2), and $\Gamma = \pi_1(M^3)$.
In other words, when $G$ is fixed to be either $D_6$ or $D_8$, is it hard to count maps $\pi_1(M^3) \to G$? 
Note that topological quantum computing with $D_6$ is known to be $\mathsf{BQP}$-universal \cite{Shawn};
it is natural to wonder if the same might be true for $D_8$.

Rather than look for homomorphisms to $G$, we might try to solve systems of equations over $G$.
Indeed, the homomorphisms from a finitely presented $\Gamma$ to $G$ can be understood as the solutions to the system of equations over $G$ with one variable for each generator of $\Gamma$ and an equation for each relation.
A general system of equations also allows for \emph{constants} from $G$ in the equations.
Goldmann and Russell showed that whenever $G$ is non-abelian, deciding existence of solutions to systems of equations is $\NP$-complete \cite{GoldmannRussell}.\footnote{
	In \cite{GoldmannRussell}, the authors also consider the special case of solving a \emph{single} equation over a group $G$.
	This case is more subtle; \emph{cf.}\ \cite{IdziakETAL} for what is currently known.
	The homomorphism problem analog of single equation satisfiability is when $\Gamma$ is a presentation with a single relation.
}
As mentioned already above, Nordh and Jonsson then showed that counting solutions to systems of equations is hard whenever $G$ is non-abelian \cite{count}.

Finally, we note that group homomomorphism problems with a fixed target $G$ can be understood as a natural generalization of graph homomorphism problems with a fixed target graph $H$ (which is itself a generalization of graph coloring problems).
Hell and Ne{\v{s}}et{\v{r}}il showed that existence problem for graph homomorphisms to $H $ is $\NP$-hard except when $H$ is bipartite or has a loop \cite{hell1990complexity}, in which cases the problem is in $\p$.
Dyer and Greenhill showed the counting problem is $\shP$ complete $H$ unless every connected component of $H$ is an isolated vertex without a loop, a complete graph with all loops present, or a complete unlooped bipartite graph \cite{DyerG00}.

\subsection{Motivations}
\label{ss:motivations}
While we believe it is natural enough to want to understand the complexity of finding and enumerating homomorphisms between groups, as suggested above, our own idiosyncratic motivations come from topological quantum computation using $(2+1)$-dimensional topological quantum field theories (TQFTs).
Briefly, the idea here is that each fixed unitary modular fusion category $\mathcal{B}$ yields---via the Reshetikhin-Turaev construction \cite{RT:paper,T:category}---a (once-extended) unitary $(2+1)$-dimensional TQFT $Z_\mathcal{B}$ where the transformations associated to mapping classes of surfaces can be used for fault-tolerant quantum computation \cite{Kitaev,Freedman}.
Morever, there exist categories $\mathcal{B}$ where such operations are \emph{$\mathsf{BQP}$-universal} for quantum computing, meaning that \emph{any} quantum computation can be encoded in them \cite{FreedmanLarsenWang}.
It is an interesting and well-known problem of ongoing research to determine exactly which modular fusion categories $\mathcal{B}$ support such universal operations, see e.g.\ \cite{Rowell:paradigm,RowellNaidu,me:dichotomy}.
We refer the reader to \cite{RowellWang} for a big picture review of such matters, aimed at mathematicians.

To understand how the present paper fits into this story, there are only two points the reader needs to accept.
First: in almost all known cases where a modular fusion category $\mathcal{B}$ supports $\mathsf{BQP}$-universal operations, the associated closed 3-manifold invariants $Z_\mathcal{B}(M) \in \CC$ are $\shP$-hard \cite{Greg,me:dichotomy}. 
Second: the simplest non-trivial family of modular fusion categories are those of the form $\mathcal{B} = DG$-mod, the category of representations of a (untwisted) Drinfeld double Hopf algebra $DG$ associated to a finite group $G$.

The first point suggests that it is reasonable to treat the problem of determing the computational complexity of the 3-manifold invariant $M^3 \mapsto Z_\mathcal{B}(M^3)$ as a kind of warm-up to the problem of understanding whether or not $\mathcal{B}$ supports $\mathsf{BQP}$-universal operations.
The second point is relevant because for such modular fusion categories, it is well-known that we have
\[ Z_{DG\text{-mod}}(M) = \frac{\#\Hom(\pi_1(M),G)}{\#G}.\]
Combining all of this, we can see the problem of understanding (for a fixed finite group $G$) the computational complexity of $M^3 \mapsto \#\Hom(\pi_1(M^3),G)$ as a step along the way to understanding the $\mathsf{BQP}$-universality of the operations supported by the TQFT associated to a modular fusion category $\mathcal{B}$.

We conclude this section by explaining briefly why \emph{3-dimensional} manifolds happen to be the focus of so much of our attention.
There is a set of group-theoretical facts that gives one explanation, and a set of TQFT-related facts that gives another.

Here are the relevant group theory facts.
For any fixed $d>3$, it is well-known there exists a polynomial time algorithm to convert a finitely presented group $\Gamma$ to a closed PL $d$-manifold $M_\Gamma^d$ with $\pi_1(M_\Gamma^d) \cong \Gamma$.
So the problem $\pi_1(M^d) \mapsto \#\Hom(\pi_1(M^d),G)$ is not significantly different from the general problem $\Gamma \mapsto \#\Hom(\Gamma,G)$.
In contrast, when $d=3$, the fundamental groups of 3-manifolds are quite constrained; see e.g.\ \cite{3manifolds}.
Going even lower, when $d=2$, the fundamental groups of $d$-manifolds are quite well understood---so much that there is an efficient algorithm to compute $\Hom(\pi_1(M^2),G)$ for any fixed finite group $G$ and any triangulated surface $M^2$; one way to see this is via ``Mednykh's formula" \cite{MednykhOG,Mednykh}.\footnote{While our work here focuses mostly on \emph{closed} manifolds, complements of knots and links in $S^3$ are another important type of manifold that are commonly used in topological quantum computation, cf.\ \cite{Kitaev,RowellNaidu,RowellWang,me:coloring}.
	With this in mind, a theorem of Eisermann is especially pertinent: for fixed nilpotent $G$, the count $K \mapsto \#\Hom(\pi_1(S^3 \setminus K), G)$ is a constant independent of the knot $K$ \cite{Eisermann}.
This result might be understood as a kind of predecessor to our Theorem \ref{thm:2}.}

The relevant TQFT facts are likewise divided between low- and high-dimensional cases.
First, in dimensions less than $2+1$, TQFT invariants of manifolds are always computable in polynomial time (for similar reasons as Mednkyh's formula, combined with the fact that the topological type of a triangulated surface can be efficiently identified).
Second, in dimensions greater than $2+1$, the consensus among experts is that most examples of (semi-simple) TQFTs are explainable via classical algebraic topological invariants, such as homotopy type, signature, \emph{etc}., and, in particular, don't look so far off from invariants of the type $M \mapsto \#\Hom(\pi_1(M),G)$.
See e.g.\ \cite{CraneYetterKauffman,Reutter,fusion2} for results making such ideas precise in the case of $3+1$-dimensional TQFTs.
In particular, for now, $(2+1)$-d TQFTs are the sweet spot where there are compelling cases for both $\mathsf{BQP}$-universality and the feasability of actually engineering them.

\subsection{Outline}
\label{ss:outline}
We prove Theorem \ref{thm:1} and Corollary \ref{cor:restrictedInput} in Section \ref{sec:thm1}, and Theorem \ref{thm:2} in Section \ref{sec:thm2}.

\section{The dichotomy for finitely presented input}
\label{sec:thm1}
In this section we prove Theorem \ref{thm:1} and Corollary \ref{cor:restrictedInput}.
The first two subsections quickly review definitions relating to counting CSPs and polymorphisms, with the goals of explaining the connection to group homomorphism counting and clearly formulating the key tool \cite[Thm.~5]{BulatovDalmau:dichotomy} (restated below as Theorem \ref{thm:maltsev}).
In the third subsection, we conclude the proof by showing that the relevant ternary relation for counting homomorphisms from a finitely presented group $\Gamma$ to a fixed finite target group $G$ admits a Mal'tsev polymorphism if and only if $G$ is abelian.
In the fourth and final subsection, we give variations of this argument that prove Corollary \ref{cor:restrictedInput}.

\subsection{Relational counting CSPs}
\label{ss:CSP}
The relational counting constraint satisfaction problem $\#\CSP(A,\calR)$ is specified by the following parameters:
\begin{itemize}
	\item a finite \emph{alphabet} (or \emph{domain}) set $A$
	\item a finite set $\calR=\{R_1,\dots,R_k\}$ of \emph{constraint types} where each $R_i$ is some $n_i$-ary relation on $A$:
		\[R_i \subseteq A^{n_i}.\]
\end{itemize}
In the case that $A$ is understood and $\calR =\{R\}$ consists of a single $n$-ary relation, we will write $\#\CSP(R)$ instead of the more notationally burdensome $\#\CSP(A,\calR)$.

An \emph{instance} $X=(V,C)$ of $\#\CSP(A,\calR)$ is a list $V=(x_1,\dots,x_k)$ of variables over $A$, together with a list $C=((C_1,\iota_1),\dots,(C_l,\iota_l))$ where each $C_i \in \calR$ is a constraint type of arity $n_i$ and $\iota_i: \{1,\dots,n_i\} \to \{x_1,\dots,x_k\}$ is a way of assigning variables from $V$ to the inputs of $C_i \subseteq R^{n_i}$.
(Note that repetitions are okay, \emph{i.e.}~ $\iota_i$ need not be injective.)
If we let $\chi_{C_i}: A^{n_i} \to \{0,1\}$ be the indicator function for $C_i$ and define
\[\begin{aligned}
	\iota_i^* \chi_{C_i}: A^k &\to \{0,1\} \\
	(a_1,\dots,a_k) &\mapsto \chi_{C_i}(a_{\iota_i(1)},\dots,a_{\iota_i(n_i)})
\end{aligned}\]
then the value of $\#\CSP(A,\calR)$ on $X$ is defined to be
\[\#\{(a_1,\dots,a_k) \in A^k \mid \iota_i^*\chi_{C_i}(a_1,\dots,a_k) = 1 \text{ for all } (C_i,\iota_i) \in C\}.\]
In plain language: the value of $\#\CSP(A,\calR)$ on $X=(V,C)$ counts the number of solutions to $X$---\emph{i.e.}~ the number of ways of assigning the variables from $V$ to elements of $A$ in such a way that every relational constraint from $C$ is satisfied.

For the purposes of Theorem \ref{thm:1}, there is really only one example we need to understand (although we will consider a new example when we consider RAAGs for Corollary \ref{cor:restrictedInput}).
Consider the finite group $G$ as an alphabet set and define the following 3-ary relation on $G$:
\[ T_G \defeq \{ (g, h, gh) \mid g, h \in G\}.\]

\begin{lem}
	\label{lem:CSP}
There is a parsimonious reduction from $\#\CSP(T_G)$ to $\#\Hom(-,G)$.
\end{lem}

\begin{proof}
Let $X=(V,C)$ be an instance of $\#\CSP(T_G)$.
We use this data to write down a finitely presented group $\Gamma_X$.
The generators of $\Gamma_X$ are the variables $x_1,\dots,x_k$ in $V$.
Each $(C_i,\iota_i) \in C$ determines a relation $r_i$ of $\Gamma_X$ as follows.
Since $\calR = \{T_G\}$ consists of a single constraint type, $C_i=T_G$ and $n_i=3$ for all $i=1,\dots,l$.
Suppose $\iota_i(j) = i_j$ for $j=1,2,3$.
Then $r_i$ is the relation
\[ x_{i_1}x_{i_2} = x_{i_3}.\]
Clearly, given $X$, we can write down $\Gamma_X$ in polynomial time.
Moreover, by construction, the set of homomorphisms $\Gamma_X \to G$ is in bijection with the solutions to $X$.
\end{proof}

\subsection{Mal'tsev polymorphisms}
\label{ss:maltsev}
Let $R \subseteq A^n$ be an $n$-ary relation on the domain $A$.
A \emph{polymorphism} of $R$ is a $d$-ary operation $f: A^d \to A$ (for some $d$) such that whenever $(x_{1,1},\dots,x_{1,n}),\dots,(x_{d,1},\dots,x_{d,n}) \in R$, it follows that
\[ (f(x_{1,1},\dots,x_{d,1}),\dots,f(x_{n,1},\dots,x_{n,d})) \in R.\]
A polymorphism is called \emph{Mal'tsev} if $d=3$ and $f(x,x,y)=f(y,x,x)=y$ for all $x,y \in A$. 
We say that $\#\CSP(A,\calR)$ admits a Mal'tsev polymorphism if there exists a simultaneous Mal'tsev polymorphism for each of the $R_i \in \calR$.

\begin{thm}[Thm.~5 of \cite{BulatovDalmau:dichotomy}]
	\label{thm:maltsev}
If $\#\CSP(A,\calR)$ does not admit a Mal'tsev polymorphism, then it is $\shP$-complete (under Turing reduction).
\qed
\end{thm}

\subsection{Concluding the proof of Theorem \ref{thm:1}}
\label{ss:concluding}
Combining Lemma \ref{lem:CSP} with Theorem \ref{thm:maltsev}, we see that the following lemma suffices to establish the hardness statement of Theorem \ref{thm:1}.

\begin{lem}
	\label{lem:maltsev}
$T_G$ admits a Mal'tsev polymorphism if and only if $G$ is abelian.
\end{lem}

\begin{proof}
	Suppose $f: G^3 \to G$ is a Mal'tsev polymorphism of $T_G$.
	Let $x,y,z$ be any group elements and  consider the following two pairs of 3-tuples of elements of $T_G$:
	\[
	\begin{array}{c|ccc}
		& g & h & gh \\
		\hline
		& x & 1 & x\\
		& x & x^{-1}y & y \\
		& zy^{-1}x & x^{-1}y & z \\
		\hline
		f: & zy^{-1}x & 1 & f(x,y,z)
	\end{array}
	\hspace{3cm}
	\begin{array}{c|ccc}
		& g & h & gh \\
		\hline
		& 1 & x & x\\
		& yx^{-1} & x & y \\
		& yx^{-1} & xy^{-1}z & z \\
		\hline
		f: & 1 & xy^{-1}z & f(x,y,z)
	\end{array}
	\]
	From the first, we see that if $f$ is a Mal'tsev polymorphism, then we need $f(x,y,z) = zy^{-1}x$.  The second shows that we need $f(x,y,z) = xy^{-1}z$.  In other words, for all $x,y,z \in G$, we need $xy^{-1}z = zy^{-1}x$, which happens if and only if $G$ is abelian. 
\end{proof}

Finally, the polynomial-time claim in the case of abelian $G$ is well-known (compare, e.g., \cite[Thm.~2.5]{me:zombies} or \cite[Thm.~7]{count}) and follows readily from the efficient computability of Smith normal form of integer matrices \cite{snf}.
Perhaps the only subtlety to note is that we should preprocess $G$ (independently of $\Gamma$) so that its prime decomposition is manifest.
Doing so ensures we do not need to worry about doing any prime factorization of the invariant factors of the abelianization of $\Gamma$---we only need to extract the primary factors that involve the primes from $|G|$. \qed

\subsection{Restricting the input group}
\label{ss:restrictedInput}
We now prove Corollary~\ref{cor:restrictedInput}.

\subsubsection*{Right-angled Artin groups}
A \emph{right-angled Artin group} (RAAG) is a finitely-presented group over a generating set $X$ for which any relations are of the form $[x,y] = 1$ for some $x,y \in X$.
A common way to specify a RAAG is using a simple graph $(V,E)$: the vertices in $V$ are taken as generators, and two such generators $v,w$ are taken to commute if and ony if $\{v,w\} \in E$.
A complete graph on $n$ vertices would yield $\ZZ^n$, while a totally disconnected graph on $n$ vertices would yield the free group $F_n$.

We again consider the finite group $G$ as an alphabet set, but now define another binary relation on $G$:
\[ R_G \defeq \{ (g, h) \mid g, h \in G, [g,h] = 1\}.\]

\begin{lem}
	\label{lem:maltsevCommutator}
	$R_G$ admits a Mal'tsev polymorphism if and only if $G$ is abelian.
\end{lem}

\begin{proof}
	Suppose $f: G^3 \to G$ is a Mal'tsev polymorphism of $R_G$.
	Let $x,y$ be any group elements and  consider the following elements of $R_G$:
	\[
	\begin{array}{c|cc}
		& g & h \\
		\hline
		& x & 1 \\
		& 1 & 1  \\
		& 1 & y \\
		\hline
		f: & x & y
	\end{array}
	\]
	Therefore, $x$ and $y$ must commute; hence, $G$ is abelian.
\end{proof}
It is now easy to check that exactly as in Lemma~\ref{lem:CSP}, there is a parsimonious reduction from $\#\CSP(R_G)$ to $\#\Hom( \mathrm{RAAGs} ,G)$, where $\mathrm{RAAGs}$ denotes the class of all finitely generated right-angled Artin groups.

\subsubsection*{Small cancellation groups.}
Let \(\Gamma = \langle S \mid R\rangle\) be a finitely presented group, where \(R\) is a symmetrized set of cyclically reduced relators, meaning that it is closed under inversion and cyclic permutations.
A \emph{piece} is a nontrivial word \(u\) that occurs as an initial segment of two distinct relators in \(R\).
Such a presentation satisfies the \emph{$C'(1/p)$ condition} if for every relator $r \in R$ and every piece \(u\) that is a subword of \(r\), one has
	\[
	|u| < \frac{1}{p}\,|r|.
	\]
It is well-known that for $p \geq 6$, the condition $C'(1/p)$ implies that Dehn's algorithm solves the word problem of the group \cite[Theorem V.4.4]{LS01} and, thus, it is word-hyperbolic \cite{Gromov}.

The proof that $C'(1/6)$ groups satisfy the hardness conclusion of Theorem \ref{thm:1} follows immediately by combining that theorem with the next lemma.

\begin{lem}\label{lem:smallCancellationGroup}
Fix a finite group $G$ and $p \in \NN$. Then there exists a polynomial time algorithm that, given a finitely presented group $\Gamma = \langle S \mid R \rangle$, constructs a new finitely presented group $\hat \Gamma$ with the following properties:
\begin{itemize}
	\item $\hat \Gamma$ satisfies the $C'(1/p)$ condition,
	\item $\hat \Gamma$ is torsion free,
	\item $\hat \Gamma$ admits a surjection onto $\Gamma$, and
	\item $\#\Hom(\hat{\Gamma},G) = \#\Hom(\Gamma,G) \cdot |G|^{2|R|}$,	where $|R|$ is the number of relators in $\Gamma$.
\end{itemize}
\end{lem}

We note the similarities between this lemma and the main results of both \cite{me:HQFT} and \cite{Clem:HQFT}, the latter of which has a special case paraphrased below in Theorem \ref{thm:clem}, and is a key part of our proof of Theorem \ref{thm:2}.

\newcommand{\Gexp}{\eta}
\begin{proof}
	Let $\Gexp$ be the exponent of $G$.

	Let $w = w(x,y)$ be a word over some two-letter alphabet $\{x,y\}$ (where $x,y$ are new symbols having nothing to do with $\Gamma$ or $G$) with the following two properties:
	\begin{itemize}
	\item there exists a $k \in \mathbb{N}$ such that no two different factors (contiguous subwords) of length $k$ in $w$ agree~-- more formally, whenever we can write $w= t_1 u v_1 = t_2 u v_2$ for words $t_1,t_2, u, v_1,v_2$, we have already $t_1 = t_2$ and $v_1 = v_2$, and
	\item the length of $w$ is at least $p\cdot k$.
	\end{itemize}
	It is well-known that for any $p$ there exist such $k$ and $w$.
	For instance take a suitable prefix of a de Bruijn sequence $B(2,k)$ for large enough $k$; see e.g.\ \cite{BerstelP07}.

	Now, given an input presentation of a group \(\Gamma = \langle S \mid R\rangle\), denote $\ell = \max\{ |r| \mid r\in R\}$.
	 For each relation $r \in R$, we add two new generators $a_r, b_r$ and replace the relation by the new one $r w(a_r^{\Gexp\ell}, b_r^{\Gexp\ell})$ (where $ w(a_r^{\Gexp\ell}, b_r^{\Gexp\ell})$ means that we substitute  $a_r^{\Gexp\ell}$ and  $b_r^{\Gexp\ell}$ for $x$ and $y$). 
	We also add the inverse of this relation and all cyclic permutations.
	 In other words,
	 \[\hat \Gamma = \langle \hat S  \mid \hat R\rangle\]
 where
	\[\hat S = S \cup \{a_r, b_r \mid r \in R\} \qquad\text{ and } \qquad\hat R = \operatorname{sm}(\{ r w(a_r^{\Gexp\ell}, b_r^{\Gexp\ell}) \mid r\in R \}),\]
	and $\operatorname{sm}()$ means that we symmetrize the relations.
	Clearly, the presentation can be computed in polynomial time. 
	 Now, any piece of $\hat R$ is either part of some of the original relations or it comes from a factor of length less than $k$ of $w$. Therefore, every piece has length at most $\Gexp\cdot \ell\cdot k$. As each of the relations in $\hat R$ has length at least $\Gexp\cdot \ell\cdot k \cdot p$ (by the choice of $w$), the condition $C'(1/p)$ follows.
	
	Note that, by construction, no relator in $\hat \Gamma$ is a proper power.
	Thus if we assume $p\ge 6$, which we may as well, we see that $\hat \Gamma$ is torsion-free.
	(Alternatively: the fact that $\hat \Gamma$ is an amalgamated product of several free groups implies it is torsion-free.)
	
	The surjection $\hat\Gamma \to \Gamma$ is straightforward by sending $\langle \{a_r, b_r \mid r \in R\} \rangle $ to $1$. 
	Finally, observe that any homomorphism $\hat{\Gamma} \to G$ must map each relator $r \in R$ to $1$ as the image of $w(a_r^{\Gexp\ell}, b_r^{\Gexp\ell})$ is clearly $1$ in $G$.
	On the other hand, the generators $\{a_r, b_r \mid r \in R\}$ can be mapped freely to $G$ with no restriction.
	Therefore, $\#\Hom(\hat{\Gamma},G) = \#\Hom(\Gamma,G) \cdot |G|^{2|R|}$.
\end{proof}

Note that we can modify the proof of Lemma~\ref{lem:smallCancellationGroup} to obtain a quotient $\tilde \Gamma$ of the group $\hat\Gamma$ such that  $\#\Hom(\tilde{\Gamma},G) = \#\Hom(\Gamma,G)$ with the price that $\tilde\Gamma$ is no longer torsion-free:
let $q> \Gexp \cdot k \cdot p \cdot \ell$ be any number coprime to the exponent $\Gexp$ of $G$ and add the relations $a_r^q = b_r^q = 1$ for all $r \in R$.
Then, clearly, any homomorphisms $\hat\Gamma \to G$ must map $a_r$ and $b_r$ to $1$ for every $r \in R$.
As each piece still has length at most $  \Gexp \cdot k \cdot \ell$ the presentation still satisfies the condition $C'(1/p)$.

Note that using Lemma~\ref{lem:smallCancellationGroup} we conclude the following corollary of \cite{ElderShenWeiss}: Let $D_n$ be any dihedral group where $n$ is not a power of two. Then the question, given a \emph{hyperbolic} group $\Gamma$ (as a finite presentation), whether there is a surjection $\Gamma \to D_{n}$ is $\NP$-complete.

\subsubsection*{Nilpotent targets allow finite domains.}

Let $F(x_1, \dots, x_n)$ denote the free group of rank $n$. 
\emph{Basic commutators} are defined inductively as follows:
each generator $x_i$ is a basic commutator of weight $1$; if $\alpha$ and $\beta$ are basic commutators of weight $k$ resp.\ $\ell$ meeting certain (syntactical) conditions\footnote{For us, these conditions are not relevant; for details we refer to \cite[Chapter 11]{Hall59}.}, then $[\alpha, \beta]$ is a basic commutator of weight $k+\ell$.
The \emph{free nilpotent group} $F(n,c)$ of rank $n$ and nilpotency class $c$ is the free group modulo the relations that all weight $c+1$ commutators (not just the basic ones) are trivial.
We will heavily rely on the following result:

\begin{thm}[see {{\cite[Theorem 11.2.4]{Hall59}}}]\label{thm:hall}
	There is a sequence $\alpha_1, \dots, \alpha_m$ of basic commutators of weight at most $c$ such that every element $g \in F(n,c)$ has a unique representation $g = \alpha_1^{e_1} \cdots \alpha_m^{e_m}$ with $e_1, \dots, e_m \in \mathbb{Z}$.
\end{thm}

Using this theorem, it follows that $F(n,c)$ also can be presented using only relations stating that all \emph{basic} commutators of weight from $c+1$ up to $2c$ are trivial (since any basic commutator of greater weight must be composed of at least one of the commutators with weight in this range)\footnote{Note that this presentation contains some redundant relations; however, it is not clear to us how to find a minimal subset of these relations to obtain a presentation of the free nilpotent group. For a further discussion on this topic, we refer to a useful MathOverflow post \cite{mo}: \url{https://mathoverflow.net/q/330277/}.}. 
This yields a presentation of $F(n,c)$ that is of polynomial size in the rank $n$ (though the size depends exponentially on $c$).

Let $G$ be a finite nilpotent group of exponent $\Gexp$ and nilpotency class $c > 1$. 
Again we modify the construction of the input group $\Gamma$ from above and denote by $n$ the number of generators of $\Gamma$.
First, we modify the free nilpotent group $F(n,c)$:
add additional relations of the form $x^\Gexp=1$ for all basic commutators of weight at most $c$; this yields a quotient $Q$ of $F(n,c)$, which is seen to be finite again using Theorem~\ref{thm:hall} (in the quotient, we have only finitely many choices for each of the exponents).
Moreover, by construction, the presentation of $Q$ at hand has polynomial size (note that, in general, the exponent of $Q$ will be greater than $\Gexp$, but this is no problem for us).
Clearly any homomorphism $\Gamma = \langle X \mid R\rangle \to G$ must factor through $Q/\langle\!\langle R\rangle\!\rangle$, which is finite. 
Thus, replacing $\Gamma$ by  $Q/\langle\!\langle R\rangle\!\rangle$ proves the last point of Corollary~\ref{cor:restrictedInput} for finite presentations.

Finally, in order to obtain a permutation representation, we can proceed as follows.
As we already know we might as well do, we will start with $\Gamma = \langle X \mid R \rangle$ where $R \subseteq \{[x,y] \mid x,y \in X\}$, \emph{i.e.}~ we assume $\Gamma$ is a presentation of a RAAG.
Apply the construction from above to obtain the finite group $F(X,c)/N$, where $N$ is the normal subgroup generated by $R$ and all $\Gexp$-th powers of basic commutators; as above, every homomorphism $\Gamma \to G$ must factor through $F(X,c)/N$.
Therefore, let us construct a permutation representation of $F(X,c)/N$.

Observe that there is an embedding $ F(X,c) \to \Pi_{S \in \binom{X}{c}} F(S,c)$ where $\binom{X}{c}$ denotes the set of all $c$-element subsets of the generators $X$.
The embedding is induced by the canonical projections $\pi_S: F(X,c) \to F(S,c)$.
By Theorem~\ref{thm:hall}, this map is injective as every basic weight $\leq c$ commutator will be non-trivial in at least one of the $F(X,c)$s.
 
Now, for each $S \in \binom{X}{c}$ let $N_S$ be the normal subgroup of $F(S,c)$ generated by the $\Gexp$-th powers of all basic commutators together with $[x,y]$ whenever $x,y \in S$ and $[x,y] = 1 \in R$.
Note that $N_S = \pi_S(N)$.

 The homomorphism $F(X,c) \to \Pi_{S \in \binom{X}{c}} F(S,c)$ also induces a homomorphism 
 \[\varphi\colon F(X,c)/N \to  \Pi_{S \in \binom{X}{c}}F(S,c)/N_S\]
  and we claim that also $\varphi$ is injective.
 Indeed, if $g \in F(X,c)$ with $g \not\in N$, then $g$ can be written as a product of powers of basic commutators where at least one of these products of powers---let us call it $\alpha^e$---is not contained in $N$. 
 By our constructions, there exist some $S \in \binom{X}{c}$ that contains all the generators appearing in $\alpha$.
 But then $\pi_S(\alpha^e)$ is not contained in $N_S$, and the same is also true for $\pi_S(g)$. 
 Therefore, $\varphi(g) \neq 1$.
 
 Finally, since $c$ is a constant, there are only constantly many different (isomorphism types of) groups $F(S,c)/N_S$.
 Therefore, for each of them we can fix a permutation representation and, thus, can compute a permutation representation of $\Pi_{S \in \binom{X}{c}} F(S,c)/N_S$ in polynomial time. 
 We obtain the desired permutation representation of $\varphi(F(X,c)/N)$ by  taking the subgroup generated by $\{\varphi(x) \mid x \in X\}$.

\section{The algorithm for class 2 targets}
\label{sec:thm2}

Fix a class 2 finite nilpotent group $G = A \times_\eta B$, where $A$ and $B$ are finite abelian groups and $\eta \in Z^2(B,A)$ is a normalized 2-cocycle.
This means that, as a set, $G = A \times B$, but the group operation is given by
\[ (a_1,b_1) \cdot (a_2,b_2) = (a_1+a_2+\eta(b_1,b_2), b_1+b_2).\]
See e.g. \cite[Ch.~IV]{Brown:book} or \cite[Ch.~17]{DummitFoote} for background.

We now prove Theorem \ref{thm:2} by establishing the existence of a polynomial time algorithm that takes a triangulation $M$ of an orientable, closed 3-manifold to the number \[\#\Hom(\pi_1(M),G).\]
Our starting point is the formula already explained in Subsection \ref{ss:strategy}:
\[ \#\Hom(\Gamma,G) = \#\Hom(\Gamma,A) \times\#\{\psi \in \Hom(\Gamma,B) \mid [\psi_* \eta] = 0 \in H^2(\Gamma, B) \}\] 
for any group $\Gamma$.

In \S\ref{ss:quad}, Lemma \ref{lem:quad} justifies the claim that $\{\psi \in \Hom(\Gamma,B) \mid [\psi_* \eta] = 0 \in H^2(\Gamma,B)\}$ can be considered the zero set of a quadratic function
\[ (-)^*\eta: \Hom(\Gamma, B) \to H^2(\Gamma,A),\]
and Lemma \ref{lem:gauss} implies that so long as $H^2(\Gamma,A)$ is bounded, we can count these zeroes efficiently.
In \S\ref{ss:alg}, we describe our overall algorithm, and in \S\ref{ss:anal}, we give further details and prove that it runs in polynomial time, thereby establishing Theorem \ref{thm:2}.

\subsection{Quadratic functions on abelian groups}
\label{ss:quad}
Let $H$ and $K$ be two abelian groups (written additively).
A function $Q: H \to K$ is called \emph{quadratic} if
\[\begin{aligned}
	\delta Q: H \times H &\to K \\
	(h_1,h_2) &\mapsto Q(h_1+h_2)-Q(h_1)-Q(h_2)
\end{aligned}\]
is bilinear.

\begin{lem}
	\label{lem:quad}
	Let $A$ and $B$ be abelian groups, and let $\eta \in Z^2(B,A)$ be a normalized 2-cocycle.
	If $\Gamma$ is any group, then the map
	\[\begin{aligned}
		(-)^*\eta: \Hom(\Gamma, B) &\to H^2(\Gamma,A) \\
		\psi &\mapsto \psi^*[\eta] = [\psi^*\eta]
	\end{aligned}\]
is quadratic.
\end{lem}

\begin{proof}
	To prove this, we will first consider how $[\eta] \in H^2(B,A)$ interacts with the addition on $B$ and use the K\"unneth isomorphism.

	Let $P: B \times B \to B$ be the addition map on $B$ ($P$ is for ``plus"), let $p_1: B \times B \to B$ be projection onto the left factor and let $p_2: B \times B \to B$ be projection on the the right factor.
	Since $B$ is abelian, $P$ is a group homomorphism.
	The K\"unneth theorem gives a splitting
	\[H^2(B \times B, A) = p_1^*(H^2(B,A)) \oplus p_2^*(H^2(B,A)) \oplus \Hom(B\otimes B, A)\]
	in which
	\[ P^*[\eta] = p_1^*[\eta] \oplus p_2^*[\eta] \oplus \beta\]
	for some unique $\beta \in \Hom(B \otimes B, A)$, \emph{i.e.}~ for some unique bilinear map $\beta: B \times B \to A$.

	Now let $\psi_1,\psi_2 \in \Hom(\Gamma, B)$.
	We identify $\psi_1 + \psi_2: \Gamma \to B$ with $P \circ (\psi_1 \times \psi_2)$ where
	\[\begin{aligned}
		\psi_1 \times \psi_2: \Gamma &\to B \times B \\
		x &\mapsto (\psi_1(x),\psi_2(x)).
	\end{aligned}\]
	Thus
	\[\begin{aligned}
		Q(\psi_1+\psi_2) &= (\psi_1+\psi_2)^*[\eta] \\
				 &= (\psi_1 \times \psi_2)^* P^* [\eta] \\
				 &= (\psi_1 \times \psi_2)^* (p_1^*[\eta] \oplus p_2^*[\eta] \oplus \beta) \\
				 &= \psi_1^*[\eta] + \psi_2^*[\eta] + (\psi_1 \times \psi_2)^* \beta \\
				 &= Q(\psi_1) + Q(\psi_2) + (\psi_1 \times \psi_2)^* \beta,
	\end{aligned}\]
	where $(\psi_1 \times \psi_2)^* \beta$ is bilinear, as needed.
\end{proof}

\begin{lem}
	\label{lem:gauss}
	Fix a positive constant $C$, and let $H$ and $K$ be finite abelian groups (encoded via primary decompositions) with $|K|<C$.
	If $Q: H \to K$ is a quadratic function (provided via oracle access), then there is a polynomial time algorithm to compute the number of zeroes of $Q$, \emph{i.e.}~
	\[ \#\{ h \in H \mid Q(h) = 0 \in K \}.\]
\end{lem}

\begin{proof}
	It is useful to introduce a few additional definitions.
	
	First, the \emph{Pontryagin dual} of $K$ is the group of characters:
	\[ \hat K \defeq \Hom(K, U(1)).\]
	Note that $\hat K \cong K$---in particular, $|\hat K| = |K|$.

	Second, if $q: H \to U(1)$ is a quadratic function, then its associated \emph{Gauss sum} is
	\[ G(q) \defeq \sum_{h \in H} q(h).\]
	It is a folklore result that for finite abelian groups $H$ provided via a primary decomposition and $q$ provided via an oracle, there is an efficient algorithm to compute $G(q)$; see \cite[Thm.~2.2]{me:TY} for a proof.

	We need one final fact, which follows from the orthogonality of characters:
	\[ \sum_{\chi \in \hat K} \chi(k) =
		\begin{cases}
			|K| & \text{ if } k=1,\\
			0 & \text{ else.}
		\end{cases}
	\]
	Therefore,
	\[\#\{h \in H \mid Q(h) = 0 \} = \frac{1}{|K|} \sum_{\chi \in \hat K} \sum_{h \in H} \chi(Q(h)).\]

	And now we have our algorithm: for each fixed $\chi$, the inner sum $\sum_{h \in H} \chi(Q(h))$ is an efficiently computable Gauss sum, since $\chi\circ Q: H \to U(1)$ is quadratic.
	Since $|K|<C$, we only have to compute a constant number of them, sum the results, and divide by $|K|$. 
\end{proof}

\subsection{The overall algorithm}
\label{ss:alg}
Here is the overall algorithm we use to establish Theorem \ref{thm:2} (the details of each step are described in the next subsection).

\begin{enumerate}
	\item Given $M$ build a new triangulated 3-manifold $N$ with three properties:
		\begin{enumerate}
			\item $\#\Hom(\pi_1(N),G)=\#\Hom(\pi_1(M),G)$,
			\item $\#H^2(N,A) = H^2(M,A)$, and
			\item $N$ is an Eilenberg-MacLane space.
		\end{enumerate}
	$\#\#$: Note that $N$ need not be homeomorphic with $M$.
\item Compute $a \defeq \#\Hom(\pi_1(N),A)$.
\item Compute $b \defeq \#\{ \psi \in \Hom(\pi_1(N),B) \mid [\psi^* \eta] = 0 \in H^2(\pi_1(N),B) \}$.
\item Return $\#\Hom(\pi_1(M),G) = a\times b$.
\end{enumerate}

\subsection{Algorithm details and running time analysis}
\label{ss:anal}
Step (1) can be achieved in polynomial time using techniques of Ennes and Maria.

\begin{thm}[After \cite{Clem:HQFT}]
	\label{thm:clem}
	If $G$ is any fixed finite group with center $A=Z(G)$, then there exists a polynomial time algorithm that converts an input triangulation $M$ of an orientable, closed 3-manifold into a new triangulation $N$ of a (new) orientable, closed 3-manifold such that
	\[\begin{aligned}
		\#\Hom(\pi_1(N),G) &= \#\Hom(\pi_1(M),G) \\
		\#H^2(N, A) &= \#H^2(M,A),
	\end{aligned}\]
	and $N$ is a hyperbolic 3-manifold.
\end{thm}

\begin{proof}
	If we did not insist that $\#H^2(N,A) = \#H^2(M,A)$, then the theorem is simply a special case of the main result of \cite{Clem:HQFT}, but rephrased in terms of triangulations instead of Heegaard diagrams.
	This rephrasing is no trouble because the authors \emph{in loc.~cit.}~ work with \emph{dense} representations of Heegaard curves, rather than normal representations.
	With a dense encoding, there are efficient algorithms to convert between triangulations and Heegaard diagrams \cite[Thm.~5.6]{Clem:HQFT}.

	As it turns out, the condition $\#H^2(N,A) = \#H^2(M,A)$ follows automatically from the construction given in \cite{Clem:HQFT}.
	We explain this now, only bothering to give the relevant topological details of the construction, with all of the impertinent complexity stripped away.
	This is easiest to explain via Heegaard splittings. 

	Recall that a \emph{Heegaard splitting} of the orientable 3-manifold $M$ is a choice of embedded, orientable genus $g$ surface $\Sigma_g \subset M$ such that $\overline{M \setminus \Sigma_g}$ is homeomorphic to the disjoint union of two genus $g$ 3-dimensional handlebodies $H_0$ and $H_1$.
	Thus we have
	\[ M = H_0 \bigsqcup_{\partial H_0 = \Sigma_g = \partial H_1} H_1.\]
	It follows from van Kampen's theorem that
	\[ i_*: \pi_1(\Sigma_g) \to \pi_1(M)\]
	is surjective.
	Morover, any homomorphism $\phi: \pi_1(M) \to G$ is uniquely determined by a homomorphism $\phi: \pi_1(\Sigma_g) \to G$ that factors through both $\pi_1(\Sigma_g) \to \pi_1(H_0)$ and $\pi_1(\Sigma_g) \to \pi_1(H_1)$. 

	To build the 3-manifold $N$ that is both hyperbolic and has the same number of homomorphisms $\pi_1(N) \to G$ as $\pi_1(M)$, the authors of \cite{Clem:HQFT} find a suitable (pointed) diffeomorphism $\Phi: \Sigma_g \to \Sigma_g$ such that for \emph{all} homomorphisms $\phi: \pi_1(\Sigma) \to G$, the action of $\Phi$ is trivial: $\Phi^*\phi = \phi$.
	They then ``twist" the identification of $\partial H_0=\Sigma_g$ and $\partial H_1=\Sigma_g$ by using $\Phi$ (instead of the identity map):
	\[ N = H_0 \bigsqcup_{\Phi: \partial H_0 \to \partial H_1} H_1.\]
	But now, using Poincar\'e duality for our orientable 3-manifolds, the fact that $A = Z(G) \le G$, and the choice of $\Phi$, we have
	\[\begin{aligned}
		H^2(N,A) &\cong H_1(N,A) \\
			&\cong \Hom(\pi_1(N),A) \\
			&\cong \{\phi: \pi_1(\Sigma_g) \to A \mid \phi \text{ factors through } H_0 \text{ and } \Phi^*\phi \text{ factors through } H_1\} \\
			&\cong \{\phi: \pi_1(\Sigma_g) \to A \mid \phi \text{ factors through both } H_0 \text{ and } H_1 \} \\
			&\cong \Hom(\pi_1(M),A) \\
			&\cong H_1(M,A) \\
			&\cong H^2(M,A).
	\end{aligned}\]
\end{proof}

Recall that the universal covering space of a (closed) hyperbolic 3-manifold is the Poincar\'e ball, which is contractible.
Thus hyperbolic 3-manifolds have trivial homotopy groups in degree 2 and higher, meaning they are Eilenberg-MacLane spaces for their fundamental groups.
(More generally, the sphere theorem implies the 3-manifolds that are Eilenberg-MacLane are exactly the same as the 3-manifolds with trivial second homotopy group $\pi_2$ and infinite fundamental group $\pi_1$.)

Step (2) of the algorithm can be performed in polynomial time for the reasons already explained at the end of Subsection \ref{ss:concluding}.

Clearly step (4) is efficient.

Finally, step (3) can be accomplished in polynomial time for the following more general reason.

\begin{thm}
	\label{thm:skeleton}
	Fix a positive constant $C$, finite abelian groups $A$ and $B$, and a normalized 2-cocycle $\eta \in Z^2(B,A)$ (where $A$ is treated as a trivial $B$ module).
	If $N$ is a connected simplicial complex with $\pi_2(N) = 0$ and $|H^2(N,A)|<C$, then there exists a polynomial time algorithm (in the number of simplices of $N$) to compute
	\[ \#\{\psi \in \Hom(\pi_1(N),B) \mid [\psi^*\eta] = 0 \in H^2(\pi_1(N),B) \}.\]
\end{thm}

Note that we only need the 3-skeleton of $N$ in what follows.

\begin{proof}
	Since $N$ is connected (\emph{i.e.}~ $\pi_0(N)=0$) and $\pi_2(N)=0$, the 3-skeleton $\Delta$ of $N$ can be completed to a (possibly infinite dimensional) CW complex $\tilde{\Delta}$ by inductively gluing on $(d+1)$-balls to kill any non-trivial homotopy groups $\pi_d(\Delta)$ for $d=3,4,5,\dots$.
	The resulting space $\tilde{\Delta}$ is an Eilenberg-MacLane space $K(\pi_1(N),1)$.
	It is well-known then that the group cohomology of $\pi_1(N)$ (with any untwisted coefficients in any abelian group $A$) is isomorphic to the cellular cohomology of $\tilde{\Delta}$ (with the same coefficients).\footnote{
		Of course, we can do the same for \emph{non-trivial} coefficient modules by equipping $\tilde{\Delta}$ with a local system.
	However, for our purposes, trivial modules are sufficient.
	Moreover, it is not clear that the conclusion of the theorem holds in the twisted setting.}
	Even better, because $\Delta$ is a 3-dimensional simplicial complex, we can identify $H^d(\pi_1(N), A)$ with the \emph{simplicial} cohomology $H^d_{simp}(\Delta, A)$ for each $d=0,1,2$, meaning we do not actually need to build $\tilde{\Delta}$ algorithmically.

	The identification $H^2(\Delta, A) \cong H^2(\pi_1(N), A)$ is more refined than simply being an isomorphism of homology groups.
	There is a natural bijection between
	\[ \{\psi \in \Hom(\pi_1(N),B) \mid [\psi^* \eta] = 0 \in H^2(\pi_1(N),A) \}\]
	and
	\[ \{[\beta] \in H^1_{simp}(\Delta,B) \mid [\beta^*\eta] = 0 \in H^2_{simp}(\Delta,A) \}.\] 
	Note that we've abused notation slightly, since $\eta: B \times B \to A$ is a group cocycle but $\beta: \Delta^1 \to B$ is a $B$-valued simplicial 1-cocycle on $\Delta$.
	(Here $\Delta^k$ is the set of $k$-cells of $\Delta$ for $k=0,1,2,3$.)
	When we write $\beta^*\eta$, what we mean is the $A$-valued simplicial 2-cocycle on $\Delta$ defined by taking
	\[ \beta^*\eta([v_0,v_1,v_2]) = \eta(\beta([v_0,v_1]),\beta([v_1,v_2])) \in A \]
	for each 2-simplex $[v_0,v_1,v_2]$ in $\Delta^2$.

	The map $\beta \mapsto \beta^*\eta$ yields a set-theoretic map
	\[ (-)^*\eta: Z^1_{simp}(\Delta, B) \to Z^2_{simp}(\Delta, A)\]
	that descends to the quadratic function
	\[ (-)^*[\eta]: H^1_{simp}(\Delta, B) \to H^2_{simp}(\Delta,A).\]
	Since $A$ and $B$ are fixed, we may assume we know their primary decompositions.
	From there, using standard techniques as in \cite{snf} or \cite{me:TY}, we can compute the primary decompositions of both $H^1_{simp}(\Delta, B)$ and $H^2_{simp}(\Delta,A)$, as well as a polynomial time algorithm to evaluate the function $(-)^*[\eta]$ with respect to these bases.
	This gives us everything we need to apply Lemma \ref{lem:gauss}.
\end{proof}

We conclude by speculating that it may be possible to avoid the use Theorem \ref{thm:clem} in the proof of Theorem \ref{thm:2} by working with the torsion linking pairing of the 3-manifold $M$.

\newcommand{\etalchar}[1]{$^{#1}$}

%%%%%%%%%%%%%%%%%%%%%%%%%%%%%%%%%%%%%%%%%%%%%%%%%%%%%%%%%%%%%%%%%%%%%%%%%%%%%%%%
\end{document}